\documentclass[oneside,11pt]{amsart}
\usepackage{amsmath}
\usepackage{amsthm}
\usepackage{pinlabel}
\usepackage{epstopdf}

\usepackage[T1]{fontenc}

\usepackage[small]{caption}

\usepackage{amsfonts}
\usepackage{amssymb}
\usepackage{amsxtra}

\newtheorem{thm}{Theorem}[section]

\newtheorem{prop}[thm]{Proposition}

\newtheorem{cor}[thm]{Corollary}

\theoremstyle{definition}
\newtheorem{defn}[thm]{Definition}
\newtheorem{rem}[thm]{Remark}

\newtheorem{quest}[thm]{Question}

\newcommand{\boundary}{\partial}

\newcommand{\bigpresentation}[2]{ \bigl\langle \, {#1} \bigm| {#2} \,
                                  \bigr\rangle }

\newcommand{\bigset}[2]{ \bigl\{ \, {#1} \bigm| {#2} \, \bigr\} }

\renewcommand{\setminus}{-}

\newcommand{\field}[1]{\mathbb{#1}}
\newcommand{\Z}{\field{Z}}
\newcommand{\R}{\field{R}}

\newcommand{\Hyp}{\field{H}}

\newcommand{\sP}{\mathcal{P}}
\newcommand{\sE}{\mathcal{E}}

\newcommand{\of}{\circ}
\renewcommand{\hat}{\widehat}

\DeclareMathOperator{\CAT}{CAT}

\DeclareMathOperator{\interior}{int}

\usepackage{combelow}

\newcommand{\Swiatkowski}{{\'{S}}wi{\k{a}}tkowski}
\newcommand{\Sierpinski}{Sierpi{\'n}ski}

\usepackage{color}
\usepackage{pinlabel}

\usepackage{ifthen}

\newcommand{\showcomments}{yes}

\newsavebox{\commentbox}
{\ifthenelse{\equal{\showcomments}{yes}}%
{\footnotemark
    \begin{lrbox}{\commentbox}
    \begin{minipage}[t]{1.25in}\raggedright\sffamily\upshape\tiny
    \footnotemark[\arabic{footnote}]}
{\begin{lrbox}{\commentbox}}}%
{\ifthenelse{\equal{\showcomments}{yes}}%
{\end{minipage}\end{lrbox}\marginpar{\usebox{\commentbox}}}
{\end{lrbox}}}

\begin{document}

\title{Nonhyperbolic Coxeter groups with Menger boundary}

\author{Matthew Haulmark}
\address{Department of Mathematics\\
1326 Stevenson Center\\
Vanderbilt University\\
Nashville, TN 37240, USA}
\email{m.haulmark@vanderbilt.edu}

\author[G.C.~Hruska]{G.~Christopher Hruska}
\address{Department of Mathematical Sciences\\
         University of Wisconsin--Milwaukee\\
         PO Box 413\\
         Milwaukee, WI 53211\\
	 USA}
\email{chruska@uwm.edu}

\author{Bakul Sathaye}
\address{Department of Mathematics and Statistics\\
Wake Forest University\\
127 Manchester Hall\\
PO Box 7388\\
Winston-Salem, NC 27109, USA}
\email{sathaybv@wfu.edu}

\begin{abstract}
A generic finite presentation defines a word hyperbolic group whose boundary is homeomorphic to the Menger curve.  In this article we produce the first known examples of non-hyperbolic $\CAT(0)$ groups whose visual boundary is homeomorphic to the Menger curve.
The examples in question are the Coxeter groups whose nerve is a complete graph on $n$ vertices for $n\geq 5$.
The construction depends on a slight extension of \Sierpinski's theorem on embedding $1$--dimensional planar compacta into the \Sierpinski\ carpet.
See the appendix for a brief erratum.
\end{abstract}



\date{\today}

\maketitle

\section{Introduction}
\label{sec:Introduction}

Many word hyperbolic groups have Gromov boundary homeomorphic to the Menger curve.  Indeed random groups have Menger boundary with overwhelming probability \cite{Champetier95,DahmaniGuirardelPrzytycki11}.  Therefore, in a strong sense, Menger boundaries are ubiquitous among hyperbolic groups.
This phenomenon depends heavily on the fact that the boundary of a one-ended hyperbolic group is always locally connected, a necessary condition since the Menger curve is a locally connected compactum.

However in the broader setting of $\CAT(0)$ groups, the visual boundary often fails to be locally connected, especially in the case when the boundary is one-dimensional.
For instance the direct product $F_2 \times \Z$ of a free group with the integers has boundary homeomorphic to the suspension of the Cantor set, which is one-dimensional but not locally connected.
The $\CAT(0)$ groups with isolated flats are, in many ways, similar to hyperbolic groups, and are often viewed as the simplest nontrivial generalization of hyperbolicity.
However, even in that setting many visual boundaries are not locally connected.  For example if $X$ is formed by gluing a closed hyperbolic surface to a torus along a simple closed geodesic loop, then its fundamental group $G$ is a $\CAT(0)$ group with isolated flats that has non--locally connected boundary \cite{MihalikRuane99}.

One might wonder whether the Menger curve boundary is a unique feature of the hyperbolic setting.
Indeed, recently Kim Ruane observed that not a single example was known of a nonhyperbolic $\CAT(0)$ group with a visual boundary homeomorphic to the Menger curve, posing the following question.

\begin{quest}[Ruane]
\label{quest:Menger}
Does there exist a nonhyperbolic group $G$ acting properly, cocompactly, and isometrically on a $\CAT(0)$ space $X$ such that the visual boundary of $X$ is homeomorphic to the Menger curve?
\end{quest}

In this article we provide the first explicit examples of nonhyperbolic $\CAT(0)$ groups with Menger visual boundary.  

\begin{thm}
\label{thm:Main}
Let $W$ be the Coxeter group defined by a presentation with $n$ generators of order two such that the order $m_{st}$ of $st$ satisfies $3\le m_{st} < \infty$ for all generators $s\ne t$
\textup{(}or more generally let $W$ be any Coxeter group whose nerve is $1$--dimensional and equal to the complete graph $K_n$\textup{)}.
\begin{enumerate}
   \item If $n=3$ the group $W$ has visual boundary homeomorphic to the circle and acts as a reflection group on the Euclidean or hyperbolic plane.
   \item If $n=4$ the group $W$ has visual boundary homeomorphic to the \Sierpinski\ carpet and acts as a reflection group on a convex subset of $\Hyp^3$ with fundamental chamber a \textup{(}possibly ideal\textup{)} convex polytope.
   \item For each $n\ge 5$, the group $W$ has visual boundary homeomorphic to the Menger curve.
\end{enumerate}
\end{thm}

The nerve of a Coxeter system is defined in Definition~\ref{defn:Nerve}.
The nerve is $1$--dimensional when all three-generator special subgroups are finite (see Remark~\ref{rem:1DNerve}).
We note that a $1$--dimensional nerve $L$ is a complete graph if and only if every $m_{st}$ is finite.

The proof of this theorem depends on work of Hruska--Ruane determining which $\CAT(0)$ groups with isolated flats have locally connected visual boundary \cite{HruskaRuaneLocalCon} and subsequent work of Haulmark on the existence of local cut points in boundaries \cite{HaulmarkCAT0}.
In particular, \cite{HaulmarkCAT0} gives a criterion that ensures the visual boundary of a $\CAT(0)$ group with isolated flats will be either the circle, the \Sierpinski\ carpet, or the Menger curve (extending a theorem of Kapovich--Kleiner from the word hyperbolic setting \cite{KapovichKleiner00}).
The circle occurs only for virtual surface groups.  In order to distinguish between the other two possible boundaries, one needs to determine whether the boundary is planar. In general the nerve of a Coxeter group does not have an obvious natural embedding into the visual boundary. However we show in this article that Coxeter groups with nerve $K_n$ do admit an embedding of $K_5$ in the boundary and hence have non-planar boundary whenever $n\geq 5$.

By Bestvina--Kapovich--Kleiner \cite{BestvinaKapovichKleiner02}, we get the following corollary.

\begin{cor}
Let $W$ be a Coxeter group with at least $5$ generators such that every $m_{st}$ satisfies $3 \le m_{st} < \infty$.
Then $W$ acts properly on a contractible $4$--manifold but does not admit a coarse embedding into any contractible $3$--manifold.
In particular, $W$ is not virtually the fundamental group of any $3$--manifold.
\end{cor}

\subsection{Related Problems and Open Questions}

A word hyperbolic special case of Theorem~\ref{thm:Main} (when all $m_{st}$ are equal and are strictly greater than $3$) is due to Benakli  \cite{Benakli92_thesis}.
Related results of Bestvina--Mess, Champetier, and Bonk--Kleiner \cite{BestvinaMess91,Champetier95,BonkKleiner05} provide various methods for constructing embedded arcs and graphs in boundaries of hyperbolic groups.

In principle, any of the well-known hyperbolic techniques could be expected to generalize to some families of $\CAT(0)$ spaces with isolated flats, although the details of such extensions would necessarily be more subtle than in the hyperbolic case.  For example, as mentioned above many groups with isolated flats have non--locally connected boundary, and thus are not linearly connected with respect to any metric.

We note that the proof of Theorem~\ref{thm:Main} given here is substantially different from the methods used by Benakli in the hyperbolic setting.
The proof here is quite short and uses a slight extension of \Sierpinski's classical embedding theorem to produce arcs in the boundary.  (We provide a new proof of this embedding theorem.)  Unlike in the hyperbolic setting, these arcs do not arise as boundaries of quasi-isometrically embedded hyperbolic planes.

Nevertheless it seems likely that many of the hyperbolic techniques mentioned above could also be extended to the present setting, which suggests the following natural questions.

\begin{quest}
What conditions on the nerve of a Coxeter group $W$ are sufficient to ensure that the open cone on the nerve $L$ admits a proper, Lipschitz, expanding map into the Davis--Moussong complex of $W$?  When does the nerve $L$ embed in the visual boundary?
\end{quest}

\begin{quest}
Let $G$ be a one-ended $\CAT(0)$ group with isolated flats.
Does the visual boundary of $G$ have the doubling property?  If the boundary is locally connected, is it linearly connected?
Note that the usual visual metrics on Gromov boundaries do not exist in the $\CAT(0)$ setting, so a different metric must be used---such as the metric studied in \cite{OsajdaSwiatkowski15}.
\end{quest}

\begin{quest}
\label{quest:QI}
Let $\mathcal{W}$ be the family of all nonhyperbolic Coxeter groups $W$ with nerve a complete graph $K_n$ where $n \ge 5$ and all labels $m_{st}=3$.
Are all groups in $\mathcal{W}$ quasi-isometric?  Can conformal dimension be used to distinguish quasi-isometry classes of groups in $\mathcal{W}$?  As above, one would need to select an appropriate metric on the boundary in order to make this question more precise.
\end{quest}

It is known that vanishing of the $\ell^2$--Betti number in dimension $i$ is a quasi-isometry invariant for each $i$ \cite{Gromov93,Pansu95_l2QI}.  Mogilski has computed the $\ell^2$--Betti numbers of the groups mentioned in the previous question: for each $W \in \mathcal{W}$ they are nontrivial in dimension two and vanish in all other dimensions \cite[Cor~5.7]{Mogilski16}.  However, this computation does not give any information about the quasi-isometry classification of the family $\mathcal{W}$.  Thus different techniques would be needed to address Question~\ref{quest:QI}.

\subsection{Acknowledgements}
During the preparation of this paper, the authors had many helpful conversations with Craig Guilbault, Mike Mihalik, Wiktor Mogilski, Boris Okun, Kim Ruane, and Kevin Schreve.  We are grateful for the insights obtained during these conversations.  The authors also thank the referees for useful feedback that has led to improvements in the exposition.

This work was partially supported by a grant from the Simons Foundation (\#318815 to G. Christopher Hruska).

\section{Arcs in the \Sierpinski\ carpet}
\label{sec:SierpinskiArcs}

In 1916, \Sierpinski\ proved that every planar compactum of dimension at most one embeds in the \Sierpinski\ carpet \cite{Sierpinski16}.  The main result of this section is Proposition~\ref{prop:StarInSierpinski}---a slight extension of \Sierpinski's theorem---which establishes the existence of embedded graphs in the \Sierpinski\ carpet that connect an arbitrary finite collection of points lying on peripheral circles.

Although \Sierpinski's proof of the embedding theorem was rather elaborate, we present here a simplified proof using the Baire Category Theorem.  The general technique of applying the Baire Category Theorem to function spaces in order to prove embedding theorems is well-known in dimension theory and appears to originate in work of Hurewicz from the 1930s.
The conclusion of Proposition~\ref{prop:StarInSierpinski} may not be surprising to experts, but we have provided the proof for the benefit of the reader.

We begin our discussion with a brief review of Whyburn's topological characterization of planar embeddings of the \Sierpinski\ carpet.

\begin{defn}[Null family of subspaces]
\label{defn:NullFamily} 
Let $M$ be a compact metric space.  A collection $\mathcal{A}$ of subspaces of $M$ is a \emph{null family} if for each $\epsilon >0$ only finitely many members of $\mathcal{A}$ have diameter greater than $\epsilon$.  If $\mathcal{A}$ is a null family of closed, pairwise disjoint subspaces, the quotient map $\pi \colon M \to M/\mathcal{A}$, which collapses each member of $\mathcal{A}$ to a point, is \emph{upper semicontinuous} in the sense that $\pi$ is a closed map (see for example Proposition~I.2.3 of \cite{Daverman86}).
\end{defn}

\begin{rem}[Planar \Sierpinski\ carpets]
\label{rem:PlanarSierpinski}
A \emph{Jordan region} in the sphere $S^2$ is a closed disc bounded by a Jordan curve.  By a theorem of Whyburn \cite{Whyburn58}, a subset $\mathcal{S}\subset S^2$ is homeomorphic to the \Sierpinski\ carpet if and only if it can be expressed as  $\mathcal{S} = S^2 - \bigcup \interior(D_i)$
for some null family of pairwise disjoint Jordan regions $\{D_1,D_2,\ldots\}$ such that $\bigcup D_i$ is dense in $S^2$.
A \emph{peripheral circle} of $\mathcal{S}$ is an embedded circle whose removal does not disconnect $\mathcal{S}$.  Equivalently the peripheral circles in a planar \Sierpinski\ carpet $\mathcal{S} \subset S^2$ are precisely the boundaries of the Jordan regions $D_i$.
We will denote the collection of peripheral circles in $\mathcal{S}$ by $\sP$.
\end{rem}

Let $E_k$ be the $k$--pointed star, i.e., the  cone on a set of $k$ points. Let $e_1,\ldots, e_k$  be the edges of $E_k$, which we will think of as embeddings of $[0,1]$ into $E$ parametrized such that $e_i(0)=e_j(0)$ for all $i,j\in\{1,..., k\}.$

\begin{prop}
\label{prop:StarInSierpinski}
Let $P_1,\ldots, P_k\in \mathcal{P}$ be distinct peripheral circles in the \Sierpinski\ carpet $\mathcal{S}$, and fix points $p_i \in P_i$. There is a topological embedding $h\colon E_k\hookrightarrow\mathcal{S}$ such that $h\circ e_i(1) = p_i$ for each $i\in\{1,\ldots, k\}$.
Furthermore the image of $E_k$ intersects the union of all peripheral circles precisely in the given points $p_1,\dots,p_k$.
\end{prop}

\begin{proof}
Let $Q$ be the quotient space $\mathcal{S}/{\sim}$ formed by collapsing each peripheral circle $P\in\sP\setminus\{P_1,\ldots,P_k\}$ to a point.
Our strategy is to first show that $Q$ is an orientable, genus zero surface with $k$ boundary curves.
Then we apply the Baire Category Theorem and the fact that $E_k$ is $1$--dimensional to find embeddings of $E_k$ that avoid the countably many peripheral points of $Q$.
The conclusion of the Proposition is illustrated in Figure~\ref{fig:Arcs}.

\begin{figure}
\labellist
\small\hair 2pt
\pinlabel $P_1$ at 509 320
\pinlabel $P_2$ at 463 520
\pinlabel $P_3$ at 338 577
\pinlabel $P_4$ at 91 486
\pinlabel $P_5$ at 218 95

\pinlabel $p_1$ at 484 240
\pinlabel $p_2$ at 355 327
\pinlabel $p_3$ at 301 545
\pinlabel $p_4$ at 87 455
\pinlabel $p_5$ at 165 158

\pinlabel $E_5$ [r] at 208 355
\endlabellist
\includegraphics[scale=0.3]{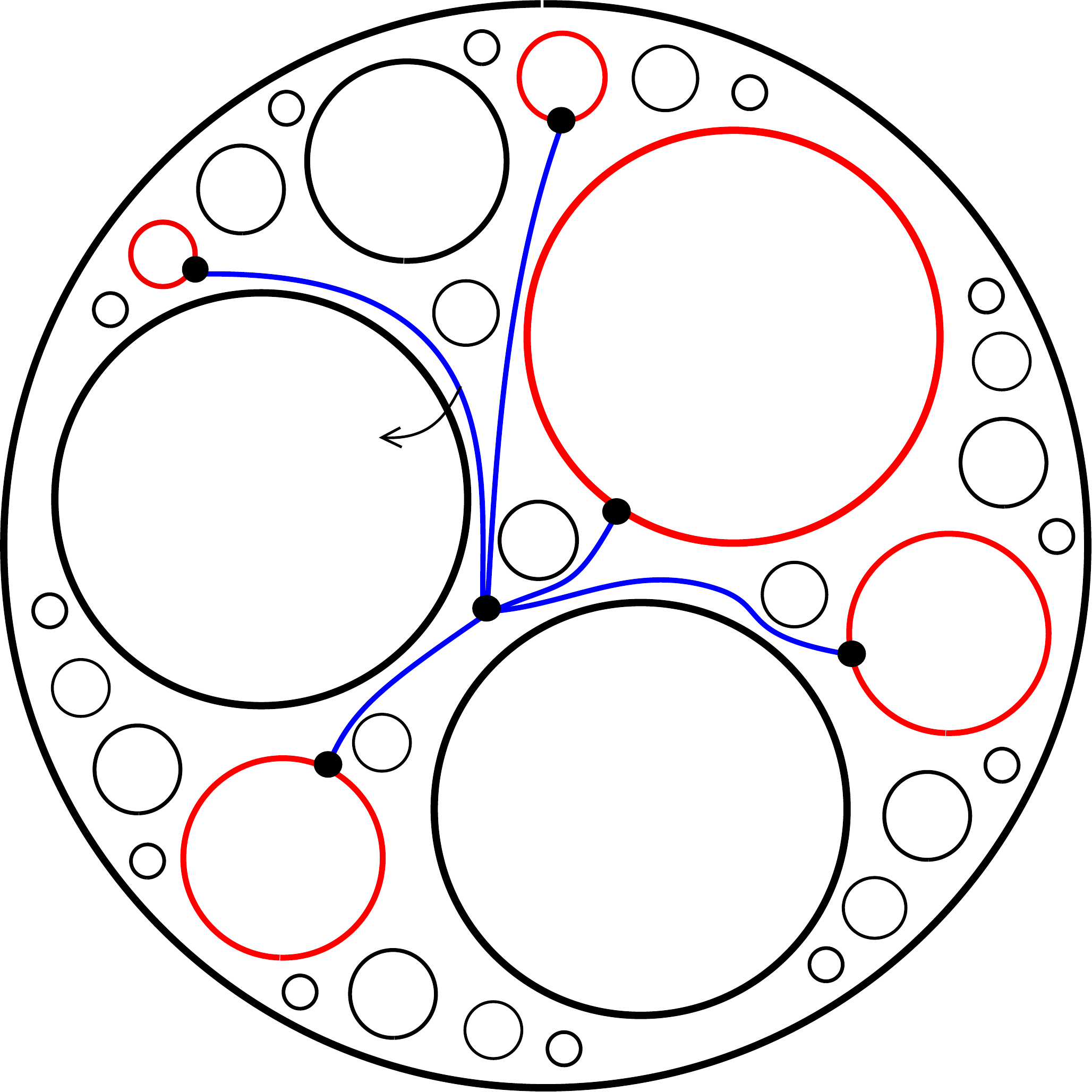}
\caption{An embedding of the graph $E_5$ that intersects the union of all peripheral circles precisely in the given points $p_1,\dots,p_5$.}
\label{fig:Arcs}
\end{figure}
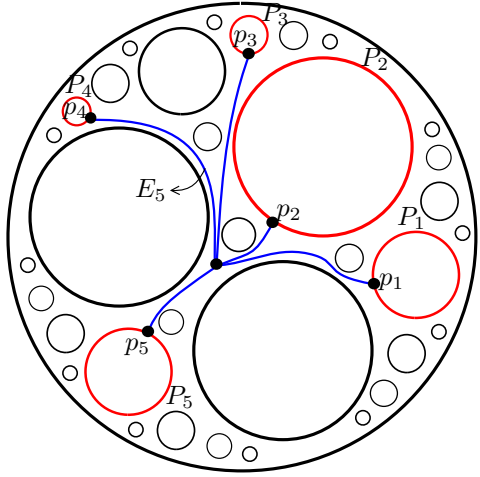

We first check that $Q$ is a surface. Fix an embedding $\mathcal{S} \hookrightarrow S^2$ as in Remark~\ref{rem:PlanarSierpinski}.  We may form $Q$ from $S^2$ in two steps as follows.  First collapse each peripheral Jordan region to a point except for those bounded by the curves $P_1,\ldots,P_k$.  By a theorem of R.L. Moore \cite{Moore25}, this upper semicontinuous quotient of $S^2$ is again homeomorphic to $S^2$.  (In particular, the quotient is Hausdorff.)  The space $Q$ may be recovered from this quotient by removing the interiors of the regions bounded by $P_1,\ldots,P_k$.  Therefore $Q$ may be obtained from a $2$--sphere by removing the interiors of $k$ pairwise disjoint Jordan regions.

Let $\pi\colon \mathcal{S}\to Q$ be the associated quotient map.
By a slight abuse of notation we let $P_i\subset Q$ and $p_i \in Q$ denote $\pi(P_i)$ and $\pi(p_i)$.  A \emph{peripheral point} of $Q$ is the image of a peripheral circle $P\in\sP\setminus\{P_1,\ldots,P_k\}$.  Observe that the peripheral points are a countable dense set in $Q$.

Let $\sE$ be the space of all embeddings $\iota\colon E_k\hookrightarrow Q$ such that for each $i$ we have $\iota \of e_i(1) = p_i$ and the image of $\iota$ intersects $P_1 \cup \cdots \cup P_k$ only in the $k$ points $p_1,\ldots,p_k$.  We fix a metric $\rho$ on $Q$ and equip $\sE$ with the complete metric given by $d(f,g)=\sup\bigset{\rho\big(f(x),g(x)\big)}{x\in E_k}$.
Our strategy is to show that for each peripheral point $p$, the set of embeddings avoiding $p$ is open and dense in $\sE$.  It then follows by the Baire Category Theorem that there exists an embedding $\iota \in \sE$ whose image contains no peripheral points.

Toward this end, we fix an arbitrary peripheral point $p\in Q$.  Since $E_k$ is compact, the set of embeddings avoiding $p$ is open, so we only need to prove that it is dense.  Suppose $f \in \sE$ and $p$ is in the image of $f$. Let $\epsilon$ be any positive number small enough that the ball $B(p,\epsilon)$ lies in the interior of $Q$.  Since $f$ is a homeomorphism onto its image, the image $f(E_k)$ is $1$--dimensional and thus does not contain any $2$--dimensional disc.  In particular, there is at least one point $q\in B(p,\epsilon)$ not in the image $f(E_k)$.
Apply an isotopy $\Phi_t$ to $Q$ keeping $Q\setminus B(p,\epsilon)$ fixed and such that $\Phi_1(q)=p$.  Then $\Phi_1\circ f\colon E_k\hookrightarrow Q$ is an element of $\sE$ which misses $p$.  Furthermore its distance from $f$ is less than $2\epsilon$. Since $\epsilon$ may be chosen arbitrarily small, we conclude that the set of embeddings avoiding $p$ is open and dense in~$\sE$.

Since the quotient map $\pi\colon \mathcal{S} \to Q$ is one-to-one on the complement of the peripheral circles, we may lift any embedding $f \in \sE$ that avoids peripheral points to an embedding $E_k \hookrightarrow \mathcal{S}$ satisfying the conclusion of the proposition.
Indeed by compactness, $f(E_k)$ is closed in $Q$, so its preimage $\pi^{-1}f(E_k)$ in $\mathcal{S}$ is closed.  The restriction of $\pi$ to this compact preimage is a continuous bijection onto the Hausdorff space $f(E_k)$, so there is a continuous inverse function $\pi^{-1}$ defined on $f(E_k)$.  The composition $\pi^{-1}f$ is the desired lift.
\end{proof}

\section{Coxeter groups and the Davis--Moussong complex}
\label{sec:Coxeter}

Let $\Upsilon$ be a finite simplicial graph with vertex set $S$ whose edges are labeled by integers $\ge 2$.  Let $m_{st}$ denote the label on the edge $\{s,t\}$. If $s$ and $t$ are distinct vertices not joined by an edge, we let $m_{st}=\infty$.
The Coxeter group determined by $\Upsilon$
is the group
\[
   W = \bigpresentation{S}{s^2, (st)^{m_{st}} \text{ for all $s,t$ distinct elements of $S$}}.
\]
A \emph{Coxeter system} $(W,S)$ is a Coxeter group $W$ with generating set $S$ as above.

\begin{defn}
\label{defn:Nerve}
The \emph{nerve} of a Coxeter system $(W,S)$ is a metric simplicial complex with a $0$--simplex for each generator $s \in S$ and a higher simplex for each subset $T \subseteq S$ such that $T$ generates a finite subgroup of $W$.
\end{defn}

If $(W,S)$ is any Coxeter system, the Coxeter group $W$ acts properly, cocompactly, and isometrically on the associated \emph{Davis--Moussong complex} $\Sigma(W,S)$, a piecewise Euclidean $\CAT(0)$ complex such that the link $L$ of each vertex is equal to  the nerve of $(W,S)$ \cite{Davis83,Moussong88,Davis_book}.

We state here a result regarding limit sets of special subgroups.  The first part is a folklore result (see, for example, \Swiatkowski\ \cite{Swiatkowski_Coxeter16}).  The second part holds for all convex subgroups of $\CAT(0)$ groups (see, for instance, Swenson \cite{Swenson99}).

\begin{prop}
\label{prop:BoundarySpecialSubgp}
Let $(W,S)$ be any Coxeter system and let $W_T$ denote the special subgroup of $W$ generated by a subset $T \subset S$. 
\begin{enumerate}
    \item
    \label{item:SpecialSubgroup}
    The Davis--Moussong complex $\Sigma(W_T,T)$ is a convex subspace of $\Sigma(W,S)$
    whose limit set $\Lambda \Sigma(W_T,T)$ is naturally homeomorphic to the visual boundary of $\Sigma(W_T,T)$.
    \item
    \label{item:Intersection}
    For any two subsets $T$ and $T'$ of $S$, we have
    \[
       \Lambda \Sigma(W_T,T) \,\cap\, \Lambda \Sigma(W_{T'},T') = \Lambda \Sigma(W_{T \cap T'},T\cap T').
    \]
\end{enumerate}
\end{prop}

\begin{rem}[$1$--dimensional nerves]
\label{rem:1DNerve}
A Coxeter group has a $1$--dimensional nerve $L$ if and only if $L$ does not contain a $2$--simplex.
A set of three generators $\{r,s,t\}$ bounds a $2$--simplex in $L$ precisely when it generates a finite subgroup, i.e., when $1/m_{rs} + 1/m_{st} + 1/m_{rt} < 1$.
(We follow the usual convention regarding $\infty$ by considering $1/m_{st}$ to equal zero when $m_{st}=\infty$.)
Therefore the nerve $L$ is $1$--dimensional if for each triangle, the sum above is $\ge 1$.
For example a Coxeter group has \emph{large type} if all $m_{st}$ satisfy $3 \le m_{st} \le \infty$.  Evidently all large type Coxeter groups have $1$--dimensional nerve.

In the $1$--dimensional case, the nerve $L$ is equal to the graph $\Upsilon$, the Davis--Moussong complex is $2$--dimensional, each face is isometric to a regular Euclidean $(2m_{st})$--sided polygon, and the nerve $L$ has a natural angular metric in which each edge $\{s,t\}$ has length $\pi - (\pi/m_{st})$. We refer the reader to \cite{Davis_book} for more background on Coxeter groups from the $\CAT(0)$ point of view.
\end{rem}

Coxeter groups of large type always have isolated flats---even when the nerve is not complete---by an observation of Wise (see \cite{Hruska2ComplexIFP} for details).
The following analogous result for Coxeter groups with nerve a complete graph follows immediately from Corollary~D of \cite{Caprace_Isolated}, since two adjacent edges in the nerve cannot both have label $2$.  

\begin{prop}
\label{prop:CompleteNerve}
Coxeter groups whose nerve is a complete graph always have isolated flats.
\end{prop}

By Hruska--Kleiner \cite{HruskaKleinerIsolated}, the groups acting geometrically on $\CAT(0)$ spaces with isolated flats have a well-defined boundary in the following sense:
If $G$ acts geometrically on two $\CAT(0)$ spaces $X$ and $Y$ with isolated flats, then there exists a $G$--equivariant homeomorphism between their visual boundaries $\boundary X$ and $\boundary Y$.
This common boundary will be denoted $\boundary G$.

\begin{prop}
\label{prop:3Boundaries}
Let $W$ be a Coxeter group whose nerve is a complete graph $K_n$ with $n \ge 3$.
The boundary $\boundary W$ of $W$ is homeomorphic to either the circle, the \Sierpinski\ carpet, or the Menger curve.
\end{prop}

\begin{proof}
A theorem due to Serre \cite[\S I.6.5]{Serre77} states that if $G$ is generated by a finite number of elements $s_1,\dots,s_n$ such that each $s_i$ and each product $s_i s_j$ has finite order, then $G$ has Serre's Property FA.  In other words, every action of $G$ on a simplicial tree has a global fixed point. Evidently $W$ satisfies Serre's criterion, and hence $W$ does not split as a nontrivial graph of groups.

Since $W$ acts geometrically on a $2$--dimensional $\CAT(0)$ space, its boundary has dimension at most $1$ by \cite{Bestvina96}.  As $W$ is infinite and not virtually free the dimension of the boundary must be exactly $1$, provided that $n \ge 3$.

The first author proves in \cite{HaulmarkCAT0} that a $\CAT(0)$ group with isolated flats with $1$--dimensional boundary that does not split over a virtually cyclic subgroup must have visual boundary homeomorphic to either the circle, the \Sierpinski\ carpet, or the Menger curve.
\end{proof}

Infinite Coxeter triangle groups always act as reflection groups on either the Euclidean plane or the hyperbolic plane.  In particular they have circle boundary.
The following proposition examines the case of Coxeter groups with nerve $K_4$.

\begin{prop}
\label{prop:W4}
Let $W$ be a Coxeter group whose nerve $L$ is a complete graph $K_4$ on $4$ vertices.
Then the boundary $\boundary W$ of $W$ is homeomorphic to the \Sierpinski\ carpet, and the limit set of each three generator special subgroup of $W$ is a peripheral circle.
\end{prop}

\begin{proof}
The nerve $L$ of $W$ is planar, so $W$ embeds as a special subgroup of a Coxeter group with visual boundary $S^2$ by a well-known doubling construction. (See, for example, \cite{DO01}.)
Indeed, one embeds $L$ into $S^2$, and then fills each complementary region in the sphere with $2$--simplices by adding a vertex in the interior of the region and coning off the boundary of the region to the new vertex.  Each such cone is ``right-angled'' in the sense that each added edge $\{s,t\}$ is assigned the label $m_{st} =2$.
This procedure produces a metric flag triangulation $\hat{L}$ of $S^2$, which has $L$ as a full subcomplex. Let $W_{\hat{L}}$ be the Coxeter group determined by the $1$--skeleton of $\hat{L}$, and having the triangulated $2$--sphere $\hat{L}$ as its nerve.  Then $\boundary W_{\hat{L}}$ is homeomorphic to $S^2$.
By Proposition~\ref{prop:BoundarySpecialSubgp}(\ref{item:SpecialSubgroup}), it follows that $\boundary W$ is planar.

Let $\mathcal{T}$ be the collection of three generator special subgroups of $W$. 
Each element $W' \in\mathcal{T}$ is an infinite triangle reflection group, i.e., either Euclidean or hyperbolic type.
By Proposition~\ref{prop:BoundarySpecialSubgp} the circle boundary of each $W' \in \mathcal{T}$ embeds in $\boundary W$, and these circles are pairwise disjoint.  Since $\boundary W$ is planar and contains more than one circle, it must be homeomorphic to the \Sierpinski\ carpet by Proposition~\ref{prop:3Boundaries}.

The group $W$ is hyperbolic relative to $\mathcal{T}$ by \cite{Caprace_Isolated}. Hung Cong Tran has shown that the Bowditch boundary is the quotient space obtained from the visual boundary $\boundary W$ by collapsing the limit sets of the three generator special subgroups and their conjugates to points \cite{Tran13}. Since $W$ has Property FA, its Bowditch boundary $\boundary(W,\mathcal{T})$ has no cut points \cite{Bowditch01}. It follows that the limit set of a three-generator special subgroup (or any of its conjugates) is always a peripheral circle of the \Sierpinski\ carpet.
\end{proof}

In fact, the group $W$ in the preceding proposition acts on $\Hyp^3$ as a geometrically finite reflection group, as described below.

\begin{rem}
In the special case where the nerve is $K_4$ and every $m_{st}=3$, the group $W$ is an arithmetic nonuniform lattice acting on $\Hyp^3$ as the group generated by the reflections in the faces of a regular ideal tetrahedron and
is commensurable with the fundamental group of the figure eight knot compliment and the Bianchi group $\textup{PGL}(2,\mathcal{O}_3)$. The relationship between $W$ and the figure eight knot group is discussed, for example, by Maclachlan--Reid (see Section~4.7.1 and Figure~13.2 of \cite{MaclachlanReid03}).

More generally each Coxeter group with nerve $K_4$ acts as a reflection group on a convex subset of $\Hyp^3$ with fundamental chamber a possibly ideal convex polytope.
Start with the triangulation $\hat{L}$ of $S^2$ described in the proof of Proposition~\ref{prop:W4}, and replace each right-angled cone on a Euclidean triangle with a single $2$--simplex.  The dual polytope $K$ has a (possibly ideal) hyperbolic metric by Andreev's theorem (see Theorem~3.5 of \cite{Schroeder09} for a detailed explanation).
The reflections in the faces of $K$ generate a Coxeter group that contains $W$ as a special subgroup.  The union of all $W$--translates of $K$ is a convex subspace of $\Hyp^3$ on which $W$ acts as a reflection group with fundamental chamber $K$.
\end{rem}
\section{Proof of the main theorem}

The goal of this section is to prove that the boundary of a Coxeter group $W$ is homeomorphic to the Menger curve when the nerve is $K_n$ for $n \ge 5$.
By Proposition~\ref{prop:3Boundaries}, it suffices to show that $\boundary W$ is nonplanar when $n\ge 5$.  Thus the following result completes the proof of Theorem~\ref{thm:Main}.

\begin{prop}
If $W$ is any Coxeter group with nerve $K_n$ for $n \ge 5$, then
the complete graph $K_5$ embeds in $\boundary W$.  In particular, $\boundary W$ is not planar.
\end{prop}

\begin{proof}
Let $W_5$ be any five generator special subgroup of $W$. By Proposition~\ref{prop:BoundarySpecialSubgp}(\ref{item:SpecialSubgroup}) it suffices to embed $K_5$ into $\boundary W_5$.

Suppose $s_1, \ldots, s_5$ are the five generators of $W_5$.  For each $i \in \{1,\dots,5\}$, let $W_4^i$ be the special subgroup of $W_5$ generated by $\{ s_1, \ldots, \hat{s_i}, \ldots, s_5 \}$. The limit set $\Lambda W_4^i$ is homeomorphic to the Sierpinski carpet by Proposition~\ref{prop:W4}.
Similarly for each $i\ne j$ in $\{1,\dots,5\}$ let $W_3^{i,j}$ denote the special subgroup generated by $\{ s_1, \ldots, \hat{s_i}, \ldots, \hat{s_j}, \ldots, s_5\}$, whose limit set is a circle.

Since $W_3^{i,j}$ is a subgroup of $W_4^i$, its limit set is a peripheral circle of the \Sierpinski\ carpet $\Lambda W_4^i$, and similarly it is a peripheral circle in the carpet $\Lambda W_4^j$.
Indeed this circle is precisely the intersection of these two \Sierpinski\ carpets by Proposition~\ref{prop:BoundarySpecialSubgp}(\ref{item:Intersection}).
The five \Sierpinski\ carpets $\Lambda W_4^i$ and their circles of intersection are illustrated in Figure~\ref{fig:K5}.

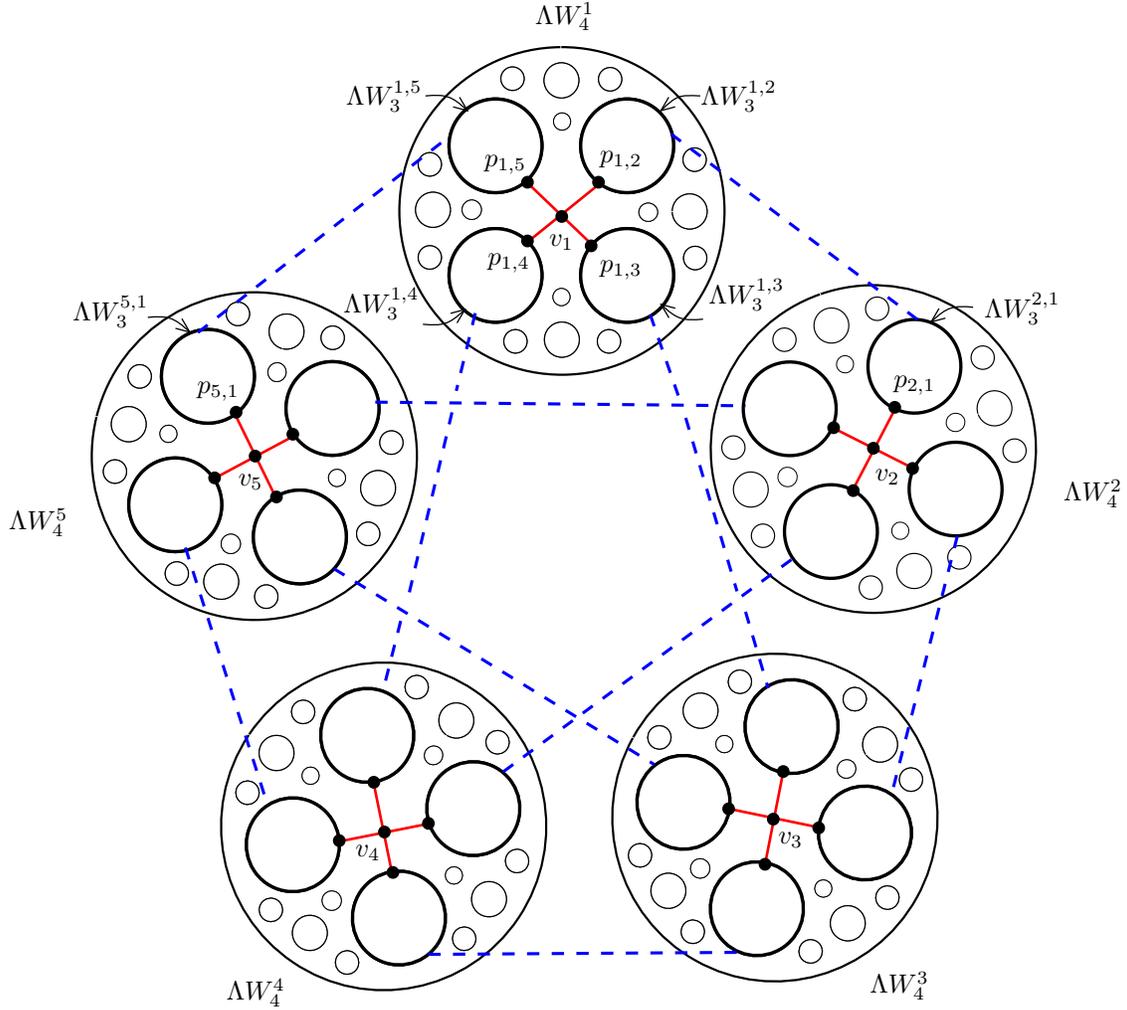
\begin{figure}
\vspace{15pt}
\labellist
\small\hair 2pt
\pinlabel $\Lambda W_4^{1}$ at 280 575
\pinlabel $\Lambda W_4^{2}$ at 591 294
\pinlabel $\Lambda W_4^{3}$ at 477 4
\pinlabel $\Lambda W_4^{4}$ at 98 0
\pinlabel $\Lambda W_4^{5}$ at -30 277

\pinlabel $v_1$ at 278 442
\pinlabel $v_2$ at 470 304
\pinlabel $v_3$ at 413 90
\pinlabel $v_4$ at 164 82
\pinlabel $v_5$ at 95 301

\pinlabel $p_{1,5}$ at 245 488
\pinlabel $p_{5,1}$ at 76 354
\pinlabel $p_{1,2}$ at 313 490
\pinlabel $p_{2,1}$ at 486 358
\pinlabel $p_{1,3}$ at 313 426
\pinlabel $p_{1,4}$ at 247 430

\pinlabel $\Lambda W_3^{1,2}$ [l] at 357 528 
\pinlabel $\Lambda W_3^{1,3}$ at 387 411
\pinlabel $\Lambda W_3^{1,4}$ at 172 406
\pinlabel $\Lambda W_3^{1,5}$ [r] at 199 528 

\pinlabel $\Lambda W_3^{5,1}$ [r] at 38 401
\pinlabel $\Lambda W_3^{2,1}$ [l] at 524 403
\endlabellist
\includegraphics[scale=0.65]{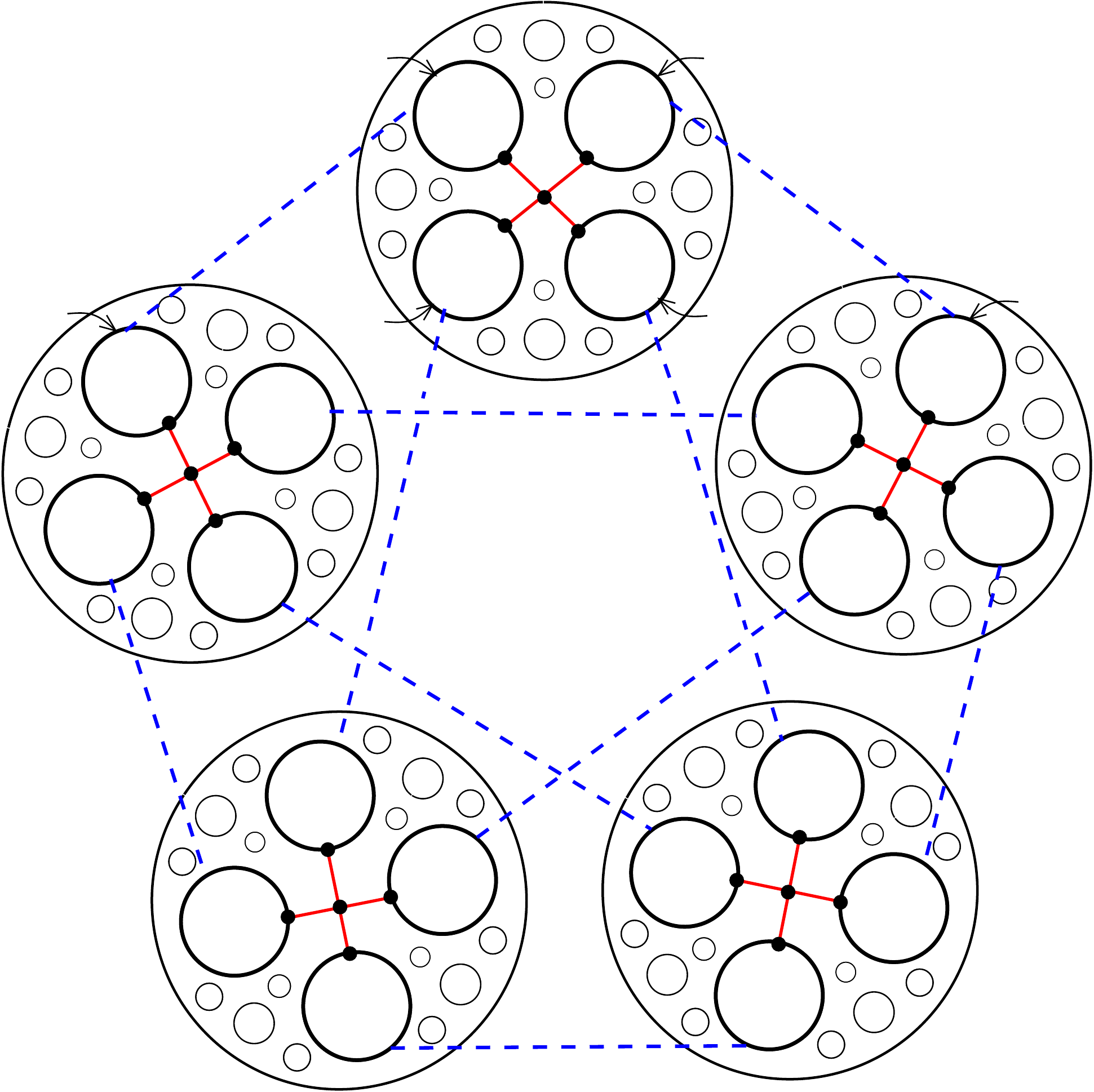}
\caption{The five \Sierpinski\ carpets $\Lambda W_4^1, \dots \Lambda W_4^5$ and their circles of intersection.  The dotted edges between carpets indicate the pairs of peripheral circles that are identified in $\boundary W_5$.}
\label{fig:K5}
\end{figure}

Choose points $p_{i,j}$ on the circles $\Lambda W_3^{i,j}$ such that $p_{i,j} = p_{j,i}$ for $i \not= j$.
Let $E_4^i$ be a collection of $4$-pointed stars for $1\leq i\leq 5.$
For a fixed $i$, we label the four edges of $E_4^i$ as $e_j^i$, where $1 \le j \le 5$ and $j\not=i$. By Proposition~\ref{prop:StarInSierpinski} for every $i$, there is a topological embedding $h_i \colon E_4^i \rightarrow \boundary W_4^i$ such that $h_i \circ e_j^i (1) =  p_{i,j} \in \boundary W_3^{i,j}$. Then $h_i \circ e_j^i(0)$ is the center of the star in $\Lambda W_4^i$, and we will denote it by $v_i$.

The union of the five stars is an embedded complete graph $K_5$ in $\boundary W_5$. Indeed, we have 5 vertices $v_1, \ldots, v_5$ and an edge between every two vertices. An edge between vertices $v_i$ and $v_j$ is given by concatenating the images of the edges $e_j^i$ and $e_i^j$ in $\Lambda W_4^i$ and $\Lambda W_4^j$ respectively.
These edges do not intersect except at their endpoints $v_i$.
\end{proof}

\appendix
\section{Erratum to ``Nonhyperbolic Coxeter groups with Menger boundary''}

The purpose of this erratum is to correct the proof of Proposition~2.3 of \cite{HHS}.  
A classical theorem of \Sierpinski\ states that every subspace of dimension at most one in the $2$--dimensional disc $D^2$ can be topologically embedded in the \Sierpinski\ carpet.
The proof of Proposition~2.3 of \cite{HHS} implicitly provides a relative version of \Sierpinski's theorem.
Unfortunately the proof of Proposition~2.3 given in \cite{HHS} is incorrect.

We provide two brief proofs that each fill this gap. One is a self-contained argument suited specifically for the needs of \cite{HHS}, and in the other we explicitly prove a relative embedding theorem that produces embeddings in the \Sierpinski\ carpet with certain prescribed boundary values.

\subsection{Introduction}
\label{sec:Introduction}

The \emph{$k$--pointed star} $E_k$ is the graph obtained as the cone of a discrete set of $k$ points $p_1,\dots,p_k$. 
Let $e_1,\dots,e_k$ be the edges of $E_k$, which we think of as embeddings of $[0,1]$ into $E_k$ parametrized such that $e_i(0) =e_j(0)$ for all $i,j\in\{1,\dots,k\}$.
Proposition~2.3 of \cite{HHS} states the following:

\begin{prop}[\cite{HHS}, Prop.~2.3]
\label{prop:StarInSierpinski}
Let $P_1,\ldots, P_k\in \mathcal{P}$ be distinct peripheral circles in the \Sierpinski\ carpet $\mathcal{S}$, and fix points $p_i \in P_i$. There is a topological embedding $h\colon E_k\hookrightarrow\mathcal{S}$ such that $h\circ e_i(1) = p_i$ for each $i\in\{1,\ldots, k\}$.
Furthermore the image of $E_k$ intersects the union of all peripheral circles precisely in the given points $p_1,\dots,p_k$.
\end{prop}

The result claimed in this proposition is correct, but the proof given in \cite{HHS} contains an error.
The purpose of this erratum is to explain the nature of this error and how to correct it.

As explained in \cite[Prop.~2.3]{HHS}, in order to prove Proposition~\ref{prop:StarInSierpinski}, it suffices to prove the following.

\begin{prop}
\label{prop:StarAvoidsDenseSet}
Let $Q$ be a compact surface of genus zero with boundary circles $P_1,\dots,P_k$, and fix points $p_i \in P_i$.  Let $T$ be a countable subset of the interior of $Q$.
Then there exists a topological embedding $h\colon E_k \hookrightarrow Q\setminus T$ such that $h\of e_i(1)=p_i$ for each $i$ and such that the image of $E_k$ intersects the boundary $\boundary Q$ precisely in the given points $p_1,\dots,p_k$.
\end{prop}

The proof of Proposition~\ref{prop:StarAvoidsDenseSet} given in \cite{HHS} involves applying the Baire Category Theorem to the space of embeddings $\mathcal{E}$ of $E_k$ into $Q$. Unfortunately, this proof is incorrect since the space $\mathcal{E}$ is not complete.  Thus one may not apply the Baire Category Theorem in this context.

\subsection{Correcting Proposition~\ref{prop:StarAvoidsDenseSet} }

This section contains two different proofs of Proposition~\ref{prop:StarAvoidsDenseSet}.
The first proof uses a simple cardinality argument.
This proof applies only to star graphs, but is completely elementary and self-contained.

The second proof is much more general; it extends a proof of \Sierpinski's embedding theorem for $1$--dimensional planar sets to give a relative embedding theorem.
The extension of \Sierpinski's result from the original setting to the relative case takes inspiration from Hatcher's exposition of Kirby's torus trick in \cite{Hatcher_TorusTrick}.

In the first proof of Proposition~\ref{prop:StarAvoidsDenseSet}, we focus on just the special case of $k=4$. This case is all that is required for the main results of \cite{HHS}. The general case follows by essentially the same reasoning.

\begin{proof}[Proof of Proposition 
\ref{prop:StarAvoidsDenseSet} ]
Let $Q$ be a compact surface of genus zero with boundary components $P_1, P_2,P_3, P_4$, and let $T$ be a countable dense set in the interior of $Q$. 
Choose an embedding $f$ of $Q$ in the Euclidean plane such that $f(Q) \subset \mathbb D^2$, the image $f(P_4)$ is the boundary circle $S^1$, and the peripheral circles $P_i$ (for $i=1,2,3$) are mapped onto the circles with radius $\frac{1}{4}$ and centers at $(-\frac{1}{2},0)$ and $(0, \pm \frac{1}{2})$. We also choose $f$ so that for $i=1,2,3$, the point $f(p_i)$ is the point  on the circle $f(P_i)$ closest to $(0,0)$ and so that $f(p_4)=(1,0)$.  For the sake of simplicity, we identify $Q$ with its image $f(Q)$ in the plane.

We wish to find an embedding of $E_4$ in $Q \setminus T \subseteq D^2$ that maps $e_i(1)$ to $p_i$ for each $i\in\{1,2,3,4\}$. 
Consider the uncountable family of embeddings $g_t\colon E_4 \to Q$, for $t \in [-\frac{1}{8}, \frac{1}{8}]$, defined as follows. For each $i\in \{1,2,3,4\}$, let $g_t(e_i)$ be the straight line segment from the point $g_t\bigl(e_i(0)\bigr)=(t,-t)$ to the point $g_t\bigl(e_i(1)\bigr)=p_i$. 
Since $T$ is countable and the sets $g_t(E_4)$ are pairwise disjoint except at the endpoints $p_i$, there are at most countably many values of $t$ such that the image of $g_t$ intersects $T$. Choose $t_0$ such that $g_{t_0}(E_k)$ is disjoint from $T$. Then $g_{t_0}$ is the desired embedding $E_4 \to Q \setminus T$.

\end{proof}

The proof above depends on the existence of a $1$--parameter family of pairwise disjoint embeddings of $E_k$ in $Q$.  As such the proof does not appear to easily generalize to embeddings of other planar graphs.  Below we discuss a different, more elaborate proof of Proposition~\ref{prop:StarAvoidsDenseSet} that holds for embeddings of a much broader family of $1$--dimensional spaces.

In the rest of this erratum, the \emph{dimension} of a normal topological space $X$ is its covering dimension, \emph{i.e.}, the supremum of all integers $n$ such that every finite open cover of $X$ admits a finite open refinement of order at most $n$.
(See Engelking \cite{Engelking_Dimensions} for details.)

\begin{prop}
\label{prop:PushOffDisc}
Let $A$ be any topological space of dimension at most $n-1$ that embeds in the closed disc $D^n$.  Let $T$ be a countable dense set in the interior of $D^n$.
Then for any embedding
  $f\colon A \to D^n$
there exists an embedding
  $g\colon A \to D^n \setminus T$
such that $f^{-1}(S^{n-1})=g^{-1}(S^{n-1})$ and $f$ and $g$ are equal on the preimage of $S^{n-1}$.
\end{prop}

\begin{proof}
In his text on dimension theory, Engelking gives a proof of \Sierpinski's embedding theorem that depends on the following key result (see \cite[Thm.~1.8.9]{Engelking_Dimensions}). For each subset $C$ of $\R^n$ with empty interior and each countable subset $T$ of $\R^n$, there exists an embedding $h\colon C \to \R^n$ such that $h(C)$ is disjoint from $T$ and the map $h$ is bounded in the sense that the Euclidean distance $d\bigl( h(c),c \bigr)$ is less than a fixed constant for all $c \in C$.

Let $f\colon A \to D^n$ be an embedding of a space of dimension at most $n-1$ in the closed disc of radius $1$.
Since $f(A)$ does not contain an $n$--dimensional disc, the image of $A$ has empty interior in $D^n$.
Let $C$ be the intersection of $f(A)$ with the interior of $D^n$, and let $T$ be any countable set in the interior of $D^n$.
As in \cite{Hatcher_TorusTrick}, we identify the interior of $D^n$ with $\R^n$ via a radial reparametrization.
Then the argument from the previous paragraph produces an embedding of $C$ into the interior of $D^n$ that misses the countable set $T$. Due to the boundedness condition, this embedding extends via the identity on $\boundary D^n$ to an embedding $\overline{h}\colon f(A) \to D^n$.
The desired embedding $g \colon A \to D^n\setminus T$ is given by the composition $g=\overline{h}\of f$.
\end{proof}

The argument above applies to an embedding in a disc relative to its boundary.  Now we show how to modify an embedding of $A$ in a compact manifold to avoid a countable set $T$, while not moving the part of $A$ that lies in the boundary of the manifold.
The following result gives an alternative proof of Proposition~\ref{prop:StarAvoidsDenseSet}.

\begin{thm}
\label{thm:PushOffSurface}
If $A$ is a space of dimension at most $n-1$ embedded in a compact $n$--manifold $M$ with boundary, and $T$ is a countable set in the interior of $M$, then $A$ can be moved to avoid the countable set, while keeping it fixed on the boundary.
\end{thm}

\begin{proof}
Cover the compact manifold $M$ with a finite collection of closed $n$--discs $\{D_i\}$ whose interiors cover the interior of $M$.  Then apply the result of Proposition~\ref{prop:PushOffDisc} to each of the discs $D_i$ to push $A$ off of the part of $T$ contained in the interior of that disc by a move that equals the identity outside of that disc.
Since the cover is finite, the composition of this sequence of moves gives an embedding with the desired properties.
\end{proof}

\bibliographystyle{alpha}
\bibliography{chruska}{}

\def\polhk#1{\setbox0=\hbox{#1}{\ooalign{\hidewidth
  \lower1ex\hbox{$\,\lhook$}\hidewidth\crcr\unhbox0}}}
  \def\RomanianComma#1{\setbox0=\hbox{#1}{\ooalign{\hidewidth
  \lower1.2ex\hbox{$\mspace{1mu}^{,}$}\hidewidth\crcr\unhbox0}}}
\begin{thebibliography}{BKK02}

\bibitem[Ben92]{Benakli92_thesis}
N.~Benakli.
\newblock {\em Poly\`{e}dres hyperboliques: passage du local au global}.
\newblock PhD thesis, Universit\'{e} Paris--Sud, 1992.

\bibitem[Bes96]{Bestvina96}
Mladen Bestvina.
\newblock Local homology properties of boundaries of groups.
\newblock {\em Michigan Math. J.}, 43(1):123--139, 1996.

\bibitem[BK05]{BonkKleiner05}
Mario Bonk and Bruce Kleiner.
\newblock Quasi-hyperbolic planes in hyperbolic groups.
\newblock {\em Proc. Amer. Math. Soc.}, 133(9):2491--2494, 2005.

\bibitem[BKK02]{BestvinaKapovichKleiner02}
Mladen Bestvina, Michael Kapovich, and Bruce Kleiner.
\newblock Van {K}ampen's embedding obstruction for discrete groups.
\newblock {\em Invent. Math.}, 150(2):219--235, 2002.

\bibitem[BM91]{BestvinaMess91}
Mladen Bestvina and Geoffrey Mess.
\newblock The boundary of negatively curved groups.
\newblock {\em J. Amer.\ Math.\ Soc.}, 4(3):469--481, 1991.

\bibitem[Bow01]{Bowditch01}
B.H. Bowditch.
\newblock Peripheral splittings of groups.
\newblock {\em Trans. Amer. Math. Soc.}, 353(10):4057--4082, 2001.

\bibitem[Cap09]{Caprace_Isolated}
Pierre-Emmanuel Caprace.
\newblock Buildings with isolated subspaces and relatively hyperbolic {C}oxeter
  groups.
\newblock {\em Innov. Incidence Geom.}, 10:15--31, 2009.

\bibitem[Cha95]{Champetier95}
Christophe Champetier.
\newblock Propri\'{e}t\'{e}s statistiques des groupes de pr\'{e}sentation
  finie.
\newblock {\em Adv. Math.}, 116(2):197--262, 1995.

\bibitem[Dav83]{Davis83}
Michael~W. Davis.
\newblock Groups generated by reflections and aspherical manifolds not covered
  by {E}uclidean space.
\newblock {\em Ann. of Math. \textup{(}2\textup{)}}, 117(2):293--324, 1983.

\bibitem[Dav86]{Daverman86}
Robert~J. Daverman.
\newblock {\em Decompositions of manifolds}, volume 124 of {\em Pure and
  Applied Mathematics}.
\newblock Academic Press, Inc., Orlando, FL, 1986.

\bibitem[Dav08]{Davis_book}
Michael~W. Davis.
\newblock {\em The geometry and topology of {C}oxeter groups}, volume~32 of
  {\em London Mathematical Society Monographs Series}.
\newblock Princeton University Press, Princeton, NJ, 2008.

\bibitem[DGP11]{DahmaniGuirardelPrzytycki11}
Fran\c{c}ois Dahmani, Vincent Guirardel, and Piotr Przytycki.
\newblock Random groups do not split.
\newblock {\em Math. Ann.}, 349(3):657--673, 2011.

\bibitem[DO01]{DO01}
Michael~W. Davis and Boris Okun.
\newblock Vanishing theorems and conjectures for the {$\ell^2$}--homology of
  right-angled {C}oxeter groups.
\newblock {\em Geom. Topol.}, 5:7--74, 2001.

\bibitem[Eng95]{Engelking_Dimensions}
Ryszard Engelking.
\newblock {\em Theory of dimensions, finite and infinite}, volume~10 of {\em
  Sigma Series in Pure Mathematics}.
\newblock Heldermann Verlag, Lemgo, 1995.

\bibitem[Gro93]{Gromov93}
M.~Gromov.
\newblock Asymptotic invariants of infinite groups.
\newblock In Graham~A. Niblo and Martin~A. Roller, editors, {\em Geometric
  group theory, Vol.~2 \textup{(}Sussex, 1991\textup{)}}, pages 1--295.
  Cambridge Univ.\ Press, Cambridge, 1993.

\bibitem[Hat13]{Hatcher_TorusTrick}
Allen Hatcher.
\newblock The {K}irby torus trick for surfaces.
\newblock Preprint. arXiv:1312.3518 [math.GT], 2013.

\bibitem[Hau18]{HaulmarkCAT0}
Matthew Haulmark.
\newblock Boundary classification and two-ended splittings of groups with
  isolated flats.
\newblock {\em J. Topol.}, 11(3):645--665, 2018.

\bibitem[HHS20]{HHS}
Matthew Haulmark, G.~Christopher Hruska, and Bakul Sathaye.
\newblock Nonhyperbolic {C}oxeter groups with {M}enger boundary.
\newblock {\em Enseign. Math.}, 65(1-2):207--220, 2020.

\bibitem[HK05]{HruskaKleinerIsolated}
G.~Christopher Hruska and Bruce Kleiner.
\newblock Hadamard spaces with isolated flats.
\newblock {\em Geom. Topol.}, 9:1501--1538, 2005.
\newblock With an appendix by the authors and Mohamad Hindawi.

\bibitem[HR]{HruskaRuaneLocalCon}
G.~Christopher Hruska and Kim Ruane.
\newblock Connectedness properties and splittings of groups with isolated
  flats.
\newblock Preprint, arXiv:1705.00784 [math.GR].

\bibitem[Hru04]{Hruska2ComplexIFP}
G.~Christopher Hruska.
\newblock Nonpositively curved 2--complexes with isolated flats.
\newblock {\em Geom. Topol.}, 8:205--275, 2004.

\bibitem[KK00]{KapovichKleiner00}
Michael Kapovich and Bruce Kleiner.
\newblock Hyperbolic groups with low-dimensional boundary.
\newblock {\em Ann. Sci. \'Ecole Norm. Sup. \textup{(}4\textup{)}},
  33(5):647--669, 2000.

\bibitem[Mog16]{Mogilski16}
Wiktor~J. Mogilski.
\newblock The fattened {D}avis complex and weighted {$L^2$}--(co)homology of
  {C}oxeter groups.
\newblock {\em Algebr. Geom. Topol.}, 16(4):2067--2105, 2016.

\bibitem[Moo25]{Moore25}
R.L. Moore.
\newblock Concerning upper semi-continuous collections of continua.
\newblock {\em Trans. Amer. Math. Soc.}, 27(4):416--428, 1925.

\bibitem[Mou88]{Moussong88}
Gabor Moussong.
\newblock {\em Hyperbolic {C}oxeter groups}.
\newblock PhD thesis, Ohio State Univ., 1988.

\bibitem[MR99]{MihalikRuane99}
Michael Mihalik and Kim Ruane.
\newblock {$\CAT(0)$} groups with non--locally connected boundary.
\newblock {\em J. London Math. Soc. \textup{(}2\textup{)}}, 60(3):757--770,
  1999.

\bibitem[MR03]{MaclachlanReid03}
Colin Maclachlan and Alan~W. Reid.
\newblock {\em The arithmetic of hyperbolic {$3$}--manifolds}, volume 219 of
  {\em Graduate Texts in Mathematics}.
\newblock Springer-Verlag, New York, 2003.

\bibitem[OS15]{OsajdaSwiatkowski15}
Damian Osajda and Jacek \Swiatkowski.
\newblock On asymptotically hereditarily aspherical groups.
\newblock {\em Proc. Lond. Math. Soc. \textup{(}3\textup{)}}, 111(1):93--126,
  2015.

\bibitem[Pan95]{Pansu95_l2QI}
P.~Pansu.
\newblock Cohomologie {$L^p$}: invariance sous quasiisom\'{e}tries.
\newblock 1995.
\newblock Preprint. \texttt{https://www.math.u-psud.fr/\~{}pansu}.

\bibitem[Sch09]{Schroeder09}
Timothy~A. Schroeder.
\newblock Geometrization of 3--dimensional {C}oxeter orbifolds and {S}inger's
  conjecture.
\newblock {\em Geom. Dedicata}, 140:163--174, 2009.

\bibitem[Ser77]{Serre77}
Jean-Pierre Serre.
\newblock {\em Arbres, amalgames, {${\rm SL}\sb{2}$}}, volume~46 of {\em
  Ast\'erisque}.
\newblock Soci\'et\'e Math\'ematique de France, Paris, 1977.
\newblock Written in collaboration with Hyman Bass.

\bibitem[Sie16]{Sierpinski16}
W.~Sierpi{\'{n}}ski.
\newblock Sur une courbe cantorienne qui contient une image biunivoque et
  continue de toute courbe donn{\'{e}}e.
\newblock {\em C.R.\ Acad.\ Sci.\ Paris}, 162:629--632, 1916.

\bibitem[Swe99]{Swenson99}
Eric~L. Swenson.
\newblock A cut point theorem for ${\CAT(0)}$ groups.
\newblock {\em J. Differential Geom.}, 53(2):327--358, 1999.

\bibitem[{\'{S}}wi16]{Swiatkowski_Coxeter16}
J.~{\'{S}}wi{\k{a}}tkowski.
\newblock Hyperbolic {C}oxeter groups with {S}ierpi\'nski carpet boundary.
\newblock {\em Bull. Lond. Math. Soc.}, 48(4):708--716, 2016.

\bibitem[Tra13]{Tran13}
Hung~Cong Tran.
\newblock Relations between various boundaries of relatively hyperbolic groups.
\newblock {\em Internat. J. Algebra Comput.}, 23(7):1551--1572, 2013.

\bibitem[Why58]{Whyburn58}
G.T. Whyburn.
\newblock Topological characterization of the {S}ierpi\'nski curve.
\newblock {\em Fund. Math.}, 45:320--324, 1958.

\end{thebibliography}

\end{document}